\documentclass[12pt, letterpaper]{amsart}

\usepackage[pagebackref,hypertexnames=false, colorlinks, citecolor=red, linkcolor=red]{hyperref}
\usepackage{amssymb}
\usepackage{mathrsfs}
\usepackage{color}
\usepackage{multirow}
\usepackage{appendix}
\usepackage{multicol}
\usepackage{array}
\usepackage{caption}
\usepackage[demo]{graphicx}
\newcolumntype{C}[1]{>{\centering\let\newline\\\arraybackslash\hspace{0pt}}m{#1}}

\normalbaselines						  %

\newcommand{\N}{{\mathbb  N}}

\newcommand{\ok}{{\overline{k}}}
\newcommand{\uk}{{\underline{k}}}
\newcommand{\on}{{\overline{n}}}
\newcommand{\un}{{\underline{n}}}
\newcommand{\ol}{{\overline{l}}}
\newcommand{\ul}{{\underline{l}}}

\newcommand{\fdot}{\,\cdot\,}

\newcommand{\wt}{\widetilde}

\newcommand{\ci}[1]{_{ {}_{\scriptstyle #1}}}

\catcode`@=11
\chardef\mathlig@atcode\count255

\def\actively#1#2{\begingroup\uccode`\~=`#2\relax\uppercase{\endgroup#1~}}
\def\mathlig@gobble{\afterassignment\mathlig@next@cmd\let\mathlig@next= }
\def\mathlig@delim{\mathlig@delim}
\def\mathlig@defcs#1{\expandafter\def\csname#1\endcsname}
\def\mathlig@let@cs#1#2{\expandafter\let\expandafter#1\csname#2\endcsname}
\def\mathlig@appendcs#1#2{\expandafter\edef\csname#1\endcsname{\csname#1\endcsname#2}}

\def\mathlig#1#2{\mathlig@checklig#1\mathlig@end\mathlig@defcs{mathlig@back@#1}{#2}\ignorespaces}

\def\mathlig@checklig#1#2\mathlig@end{%
	\expandafter\ifx\csname mathlig@forw@#1\endcsname\relax
	\expandafter\mathchardef\csname mathlig@back@#1\endcsname=\mathcode`#1%
	\mathcode`#1"8000\actively\def#1{\csname mathlig@look@#1\endcsname}%
	\mathlig@dolig#1\mathlig@delim
	\fi
	\mathlig@checksuffix#1#2\mathlig@end
}

\def\mathlig@checksuffix#1#2\mathlig@end{%
	\ifx\mathlig@delim#2\mathlig@delim\relax\else\mathlig@checksuffix@{#1}#2\mathlig@end\fi
}
\def\mathlig@checksuffix@#1#2#3\mathlig@end{%
	\expandafter\ifx\csname mathlig@forw@#1#2\endcsname\relax\mathlig@dosuffix{#1}{#2}\fi
	\mathlig@checksuffix{#1#2}#3\mathlig@end
}

\def\mathlig@dosuffix#1#2{%
	\mathlig@appendcs{mathlig@toks@#1}{#2}%
	\mathlig@dolig{#1}{#2}\mathlig@delim
}

\def\mathlig@dolig#1#2\mathlig@delim{%
	\mathlig@defcs{mathlig@look@#1#2}{%
		\mathlig@let@cs\mathlig@next{mathlig@forw@#1#2}\futurelet\mathlig@next@tok\mathlig@next}%
	\mathlig@defcs{mathlig@forw@#1#2}{%
		\mathlig@let@cs\mathlig@next{mathlig@back@#1#2}%
		\mathlig@let@cs\checker{mathlig@chck@#1#2}%
		\mathlig@let@cs\mathligtoks{mathlig@toks@#1#2}%
		\expandafter\ifx\expandafter\mathlig@delim\mathligtoks\mathlig@delim\relax\else
		\expandafter\checker\mathligtoks\mathlig@delim\fi
		\mathlig@next
	}%
	\mathlig@defcs{mathlig@toks@#1#2}{}%
	\mathlig@defcs{mathlig@chck@#1#2}##1##2\mathlig@delim{%
		\ifx\mathlig@next@tok##1%
		\mathlig@let@cs\mathlig@next@cmd{mathlig@look@#1#2##1}\let\mathlig@next\mathlig@gobble
		\fi
		\ifx\mathlig@delim##2\mathlig@delim\relax\else
		\csname mathlig@chck@#1#2\endcsname##2\mathlig@delim
		\fi
	}%
	\ifx\mathlig@delim#2\mathlig@delim\else
	\mathlig@defcs{mathlig@back@#1#2}{\csname mathlig@back@#1\endcsname #2}%
	\fi
}%

\catcode`@\mathlig@atcode

\mathchardef\ordinarycolon\mathcode`\:
\def\vcentcolon{\mathrel{\mathop\ordinarycolon}}
\mathlig{:=}{\vcentcolon=}
\mathlig{::=}{\vcentcolon\vcentcolon=}

\numberwithin{equation}{section}

\theoremstyle{plain}
\newtheorem{theo}{Theorem}[section]

\newtheorem{lem}[theo]{Lemma}

\theoremstyle{definition}

\theoremstyle{remark}
\newtheorem*{ex*}{Example}
\theoremstyle{remark}
\newtheorem*{exs*}{Examples}
\theoremstyle{remark}
\newtheorem*{rem*}{Remark}
\newtheorem{rem}[theo]{Remark}
\newtheorem*{rems*}{Remarks}

\usepackage{array}
\newcolumntype{C}[1]{>{\centering\let\newline\\\arraybackslash\hspace{0pt}}m{#1}}

\title[$X_2$ Moment Representations]{Moment Representations of Type I $X_2$ Exceptional Laguerre Polynomials}
\author{Constanze~Liaw}
\address{CASPER and Department of Mathematics, Baylor University, One Bear Place \#97328,      
 Waco, TX  76798, USA}
\email{Constanze$\underline{\,\,\,}$Liaw@baylor.edu}
\urladdr{http://sites.baylor.edu/constanze$\underline{\,\,\,}$liaw/}
\thanks{This work was supported by a grant from the Simons Foundation (\#426258, Constanze Liaw).}

\author{Jessica~Stewart~Kelly}
\address{Department of Mathematics, Christopher Newport University, 1 Avenue of the Arts, Newport News, VA 23606}
\email{Jessica$\underline{\,\,\,}$Kelly@cnu.edu}

\author{John~Osborn}
\address{Department of Mathematics, Baylor University, One Bear Place \#97328,      
 Waco, TX  76798, USA}
\email{John$\underline{\,\,\,}$Osborn@baylor.edu}

\begin{document}

\date{\today}
\subjclass[2010]{33C45, 34B24, 42C05, 44A60.}
\keywords{exceptional orthogonal polynomials, moments.}

\begin{abstract}
	The $X_m$ exceptional orthogonal polynomials (XOP) form a complete set of eigenpolynomials to a differential equation.  Despite being complete, the XOP set does not contain polynomials of every degree. Thereby, the XOP escape the Bochner classification theorem.
	
	In literature two ways to obtain XOP have been presented. When $m=1$, Gram--Schmidt orthogonalization of a so-called ``flag" was used. For general $m$, the Darboux transform was applied.
	
	Here, we present a possible flag for the $X_m$ exceptional Laguerre polynomials. We can write more about this. We only want to make specific picks when we also derive determinantal representations. There is a large degree of freedom in doing so. Further, we derive determinantal representations of the $X_2$ exceptional Laguerre polynomials involving certain adjusted moments of the exceptional weights. We find a recursion formula for these adjusted moments. The particular canonical flag we pick keeps both the determinantal representation and the moment recursion manageable.
\end{abstract}

\maketitle
\section{Introduction}
Exceptional orthogonal polynomial (XOP) systems are a generalization of the classical orthogonal polynomial systems.  The classical polynomial systems of Laguerre, Jacobi, and Hermite were shown to be the only polynomial systems of their kind in Bochner's well-known classification theorem \cite{KMUG, KMUG1}.  In addition to satisfying the requirements to be a classical orthogonal polynomial system, XOP systems require that there be an associated finite set, $A \subset \mathbb N_0$, for which the associated second-order differential equation has no polynomial solution of degree $n$ for $n \in A$. An XOP is designated as $X_m$ where $m$ is the cardinality of $A$; $m$ is called the ``co-dimension'' of the XOP system. 

An intense study of these XOP systems began less than a decade ago and originated in quantum mechanics as an alternative to the Lie-algebraic approach to exactly solvable models \cite{Dutta-Roy, grandati11, grandati12, Quesne, Quesne2}. From a purely mathematical perspective, these XOP families are of interest for their relationship to classical orthogonal polynomials and associated properties, such as their spectral analysis, location of zeros, intertwining nature, and recursion formulas \cite{KMUG9, KMUG3, KMUG4, KMUG5, KMUG6, Zeros,  HoOdakeSasaki, HoSasakiZeros2012,  Midya-Roy, Odake-Sasaki1}.  A complete classification for the XOP families has also been given \cite{Classification} and up to a linear transformation of variable, the XOP families include: Types I, II, and III Laguerre; Types I and II Jacobi; and Hermite.  

The focus of this work will be on the determinantal representations for the Type I Laguerre XOP systems of co-dimension two, which we abbreviate as $X_2^I$-Laguerre.  Some results, such as Lemma \ref{lemma}, will be applicable for any co-dimension.  Determinantal representations for the classical orthogonal polynomials have been well-studied \cite{Chihara}. For the determinantal representation of the XOP, work has been restricted to the $X_1$-Laguerre and Jacobi cases.  In \cite{Liaw-Osborn}, two determinantal representations for the $X_1^I$ Laguerre XOP system were presented.  The first used the canonical form for the moments; the second simplified computation by using adjusted moments.  The formation of the exceptional operator can be approached via a Darboux transform applied on the classical operator. By focusing on the underlying structure induced by the Darboux transform, the earlier result of Liaw and Osborn \cite{Liaw-Osborn} is extended in this article.  In \cite{Kelly-Liaw-Osborn}, a universal criterion for the exceptional condition and generalized determinantal representation for all types of $X_1$-Laguerre and Jacobi XOP systems was shown via the adjusted moments previously introduced.  

As mentioned in \cite{Kelly-Liaw-Osborn}, extending the determinantal results to higher co-dimension, that is, the $m\ge 2$ case, becomes difficult.  The complications arise as increasing the co-dimension $m$ results in an increase in the number of choices one can make regarding the associated flag.  These choices become more apparent in Section \ref{flag}.  In addition, the use of the adjusted moments as seen in \cite{Liaw-Osborn} greatly simplifies this process.  

\subsection{Outline}
In Section \ref{s-prelims} we introduce the $X_m^I$-Laguerre differential system including the polynomials and the weight function. We define and briefly discuss an alteration to the exceptional moments, which we call the \textit{adjusted moments}. In Section \ref{flag}, we explain how the exceptional conditions are responsible for excluding polynomials of degrees $0, \hdots, m-1$ from being eigenfunctions. In a sense, the main idea behind obtaining a determinantal representation for the $X_2^I$-Laguerre polynomials is to merge the exceptional conditions with the Gram--Schmidt algorithm. Namely, the exceptional conditions inform us to which \emph{flag} we apply Gram--Schmidt. The determinantal representations are the topic of Section \ref{s-determinantal}. They provide a way to find a general $X_2^I$-Laguerre polynomial that does not explicitly go through the Darboux transform or classical orthogonal polynomials, but rather uses the exceptional conditions and the adjusted moments. Recursion formulas for computing the ``matrix" of adjusted moments are presented (Section \ref{s-moments}) along with expressions for the initial adjusted moments, that is, those needed to start the recursion (Section \ref{App:AppendixA}).

\section{Preliminaries}\label{s-prelims}
The primary focus of our study will be on the $X_m^I$-Laguerre polynomials.  We refer the reader to the literature for a more in-depth look at the Laguerre XOP families.  In particular, a spectral study of the $X_1$-Laguerre polynomials \cite{Atia-Littlejohn-Stewart}. The origination of the Type III Laguerre along with a comprehensive look at all three types of exceptional Laguerre OPS may be found in \cite{LLMS}.

\begin{rem}
	In order to distinguish between the classical Laguerre orthogonal polynomial system and the XOP system, we introduce the ``hat'' notation.  The ``hat'' will appear on expressions, equations, polynomials, etc. which are associated with the exceptional case.  For example, the classical Laguerre polynomial of degree $n$ with parameter $\alpha$ is denoted by $L_n^{\alpha}(x)$; while our $X_m^I$-Laguerre polynomial of degree $n$ with parameter $\alpha$ is denoted by $\widehat L_{m,n}^{I,\alpha}$.
\end{rem}

The second-order differential expression associate with the $X_m^I$-Laguerre XOP system is given by
\begin{align}\label{eq: xop-de}
 \widehat \ell_m^{I,\alpha}[y](x)&:=-xy''(x)+\left(x-\alpha-1+2x\left(\log L_m^{\alpha-1}(-x)\right)'\right)y'(x)\\ \nonumber
 & \quad \quad +\left(2\alpha\left(\log L_m^{\alpha-1}(-x)\right)'-m\right)y(x)\,. 
\end{align} We draw the reader's attention to the logarithmic derivative of the classical Laguerre polynomial, $L_m^{\alpha-1}(-x)$, which appears in the coefficients of the $y'$ and $y$ terms.  Unlike its classical counterpart, XOP systems have  expressions which involve non-polynomial coefficients.  These coefficients will provide the requirements for what is referred to in Section \ref{flag} as the ``exceptional condition.''

The $X_m^I$-Laguerre polynomials $\left\{\widehat L_{m,n}^{I,\alpha}\right\}_{n=m}^\infty$ satisfy the Sturm-Liouville eigenvalue equation given by
\begin{equation}\label{eq: de}
\widehat \ell_m^{I,\alpha}[y](x)=(n-m)y(x)\, \quad (0<x<\infty)\,.
\end{equation} The $X_m^I$-Laguerre XOP system includes polynomials of degree $n \geq m$.  

Solutions to the eigenvalue equation \eqref{eq: de} can be characterized in a variety of ways, see \cite{LLMS}.  In particular, there exists the following relationship between the classical and exceptional $X_m^I$-Laguerre polynomials:
\[\widehat L_{m,n}^{I, \alpha}(x)=L_m^\alpha (-x)L_{n-m}^{\alpha-1}(x)+L_m^{\alpha-1}(-x)L_{n-m-1}^\alpha (x) \quad (n\geq m).
\]

In addition, the exceptional $X_m^I$-Laguerre polynomials $\left\{\widehat L_{m,n}^{I,\alpha}\right\}_{n=m}^\infty$ satisfy an orthogonality condition on $(0,\infty)$ for $\alpha >0$ with respect to the weight function 
\begin{equation}\label{eq: weight}
\widehat W_m^{I,\alpha}=\frac{x^\alpha e^{-x}}{(L_m^{\alpha-1}(-x))^2} \quad (0<x<\infty).
\end{equation}
Remarkably, for any finite co-dimension $m$, the $X_m^I$-Laguerre polynomials are complete in $L^2((0,\infty); \widehat W_m^{I,\alpha})$, see e.g.~\cite{KMUG6}.

In order to simplify our calculations for the determinantal representations for the $X_2$-Laguerre XOP, we choose to use adjusted moments.  We define these adjusted moments as 
\begin{equation}\label{adjmoments}
\widetilde \mu_{i,j}=\int_0^\infty (x-r)^i(x-s)^j\widehat W_m^{I,\alpha}(x)\,dx\,,
\end{equation} where $r$ and $s$ are the two distinct roots of $L_2^{\alpha-1}(-x)$ and $\widehat W(x)$ is an abbreviation for the $X_2$-Laguerre polynomials.  We describe the importance of these roots $r$ and $s$ in Section \ref{flag} below.

\section{The Flag}\label{flag}
In this section, we characterize the subspace spanned by the first $n$ of the $X_m^I$-Laguerre polynomials as those which satisfy the following $m$ exceptional conditions:
\begin{equation}\label{eq: exceptional condition}
	\xi_jy'(\xi_j)+\alpha y(\xi_j)=0 \text{ for } j=1, 2, \ldots, m,\;
\end{equation}
and where $\left\{\xi_j\right\}_{j=1}^m$ denote the $m$ real roots of the classical Laguerre polynomial $L_m^{\alpha-1}(-x)$.  As the roots of $L_m^{\alpha-1}(-x)$ are simple, these $m$-exceptional conditions are not redundant.  Recall that the exceptional polynomials, belonging to the $X_m^I$-Laguerre XOP sequence, will be of consecutive degrees beginning with degree $m$.
  
To tie these exceptional conditions in with the exceptional polynomials, two definitions are necessary: define the span of the first $k+1$ polynomials of the $X_m^I$-Type I Laguerre exceptional orthogonal polynomial system to be
\[
\mathcal L_{m,k}:=\text{span}\left\{\widehat
L_{m,j}^{I,\alpha}: j=m, m+1, \ldots, m+k\right\}\,.
\]

Let $\mathcal P_{k}$ represent the set of all polynomials whose degree is at most $k$.  Then define  $$\mathcal M_{m,k}:=\left\{p \in \mathcal P_{m+k}: p\text{ satisfies } \eqref{eq: exceptional condition}\right\}.$$

\begin{lem}
	\label{lemma}
	The sets $ \mathcal L_{m,k} = \mathcal M_{m,k} $ for all $m\in  \mathbb N$ and $k\in  \mathbb N_0$.
\end{lem}

\begin{rem} Before proving Lemma \ref{lemma}, we note that the original dimension argument used in the proof to find the equality of these sets was first introduced in \cite{LLMS}, while the entire proof is analogous to the proofs for \cite[Proposition 5.3]{Recurrence}, \cite[Lemma 4.1]{Kelly-Liaw-Osborn}, and \cite[Lemma 2.1]{Liaw-Osborn}.  As an adjustment is required to account for the $m$ exceptional conditions, the proof is included.
\end{rem}

\begin{proof} 
Fix $m \in \mathbb N$.  We will show the inclusion of subspaces $\mathcal L_{m,k}\leq \mathcal M_{m,k}$ for all $k \in \N_0$.  Let $y \in \mathcal L_{m,k}$.  Since $\mathcal L_{m,k}$ is invariant under $\widehat \ell_m^{I,\alpha}$,  it follows that $\widehat \ell_m^{I, \alpha}[y]\in \mathcal L_{m,k}$.  In particular, this requires that $\widehat \ell_m^{I, \alpha}[y]$ is polynomial---this will occur if and only if \eqref{eq: exceptional condition} is satisfied.  Thus, $y \in \mathcal M_{m,k}$ and $\mathcal L_{m,k}\leq \mathcal M_{m,k}$

On the other hand, a dimension argument is used to show that $\mathcal M_{m,k}\leq \mathcal L_{m,k}$.  To find the dimension of $\mathcal M_{m,k}$, observe that $P_{m+k}$ has dimension $m+k+1$ and that there are $m$ exceptional conditions imposed on $P_{m+k}$ in order to form $\mathcal M_{m,k}$.  Therefore $\dim\mathcal M_{m,k}=k+1< \infty$.  Clearly, $\dim \mathcal L_{m,k}=k+1$ as it is spanned by $k+1$ linearly independent polynomials.  As $\mathcal L_{m,k}$ and $\mathcal M_{m,k}$ are both subspaces of $\mathcal P_{m+k}$ and $\mathcal M_{m,k} \leq \mathcal L_{m,k}$, the result follows.
\end{proof}

Lemma \ref{lemma} is applicable to the $X_m^I$-Laguerre orthogonal polynomial systems for all $m \in \mathbb N$. As discussed in \cite[Lemma 4.1]{Kelly-Liaw-Osborn}, the computational aspects related to exceptional orthogonal polynomial systems of higher order become difficult.  With this in mind, we now restrict ourselves for the remainder of the paper to $m=2$.  Let the two roots of $L_2^{\alpha-1}(-x)$ be denoted as $r$ and $s$.

With the exceptional roots $r=-(\alpha+1)-\sqrt{\alpha+1}$ and $s=-(\alpha+1)+\sqrt{\alpha+1}$, we define degree $k$ polynomials:

For $k=2$, \begin{align}\label{v_2}
	v_2(x)&=L_2^\alpha(-x)\\ \nonumber
	&=\frac{1}{2}(x-r)(x-s)+x-r-\sqrt{\alpha+1}\,,
	\end{align}
for $k=3$,  \begin{align}\label{v_3}
	v_3(x)=(x-r)^2(x-s+1)\,,
	\end{align}
and for $k\ge 4$,
\begin{align}\label{v_k}
v_k(x) := (x-r)^{\overline k}(x-s)^{\underline k}\,,
\end{align}
where we use the floor and ceiling functions $\lceil \fdot \rceil$ and $\lfloor\fdot\rfloor$ (respectively) in the abbreviated notation $\ok:=\lceil k/2 \rceil$ and $\uk:=\lfloor k/2 \rfloor$.

\begin{lem}\label{l-flag}
	The sequence of polynomials $\{v_2, v_3, v_4, \ldots\}$ forms a flag for $\widehat \ell_2^{I,\alpha}$. 
\end{lem}

\begin{rem}
	We recall that, informally, in the case of XOP systems, a \emph{flag} is a sequence of polynomials whose span is preserved under an exceptional operator.  
\end{rem}

\begin{proof} 
By definition of $v_k$, it is ensured that only polynomials of degree 0 and 1 will be excluded.  As a result of Lemma \ref{lemma}, it is enough to show that the $v_k$ satisfy the exceptional conditions outlined in \eqref{eq: exceptional condition}.  These exceptional conditions reduce in the $X_2^I$-Laguerre case to: 
\begin{align}\label{eq. X2 exceptional conditions}
	r y'(r)+\alpha y(r)&=0\\ \nonumber
	s y'(s)+\alpha y(s) &=0\,.
\end{align}

Here $v_2$, the first element of the flag, is a classical Laguerre polynomial.  We could use any polynomial that satisfies the two exceptional conditions of \ref{eq. X2 exceptional conditions} for $v_3$, however, \eqref{v_3} seems to be the simplest possible, in terms of its representation using the two exceptional roots. 

It is straightforward to check that $v_2$ and $v_3$ both satisfy the exceptional conditions.  For $k\geq 4$, we have that both $v_k(r)=0$ and $v_k(s)=0$; therefore, it is enough to show that $v_k'(r)=0$ and $v_k'(s)=0$.  This too follows easily as $\overline k, \underline k \geq 2$.
\end{proof}
\section{Determinantal Representations}\label{s-determinantal}
We use the adjusted moments defined in \eqref{adjmoments} to provide a determinantal representation formula for the $X_2^I$-Laguerre orthogonal polynomials. 

This is done by expressing these polynomials in terms of powers of the the terms $(x-r)$ and $(x-s)$, where $r$ and $s$ are the distinct roots of the generalized classical Laguerre polynomial $L_2^{\alpha-1}(-x)$. 
Fix $n\in \N$ and --- as Ansatz --- consider the expansion
\begin{align}\label{e-Ansatz}
\widehat L_{2,n}^{I,\alpha}(x) 
= 
\sum_{k=0}^n
a_{nk} (x-r)^\ok (x-s)^\uk,
\end{align}
where $\ok$ and $\uk$ are as indicated previously in Section \ref{flag}, that is, $\ok=\lceil k/2 \rceil$ and $\uk=\lfloor k/2 \rfloor$.

The main idea behind the determinantal representation is to determine the $n+1$ coefficients $a_{nk}$ for $k=0, \hdots, n$. This is accomplished through working with the linear system $Ma = b$ with  $(n+1)\times (n+1)$ matrix $M$, as well as
\begin{align}\label{d-b}
a:=\begin{pmatrix}
a_{n0}\\
a_{n1}\\
\vdots\\
a_{nn}
\end{pmatrix}
\qquad
\text{and}
\qquad
b:=\begin{pmatrix}
0\\
\vdots\\
0\\
K_n
\end{pmatrix},
\end{align}
where $K_n = \|\widehat L_n^{I,\alpha}\|^2$ depends on the normalization convention. For us it merely matters that $K_n$ is non-zero.

\begin{theo}\label{thm: Matrix}
	The $X_2^I$-Laguerre polynomials can be obtained through two representations
	\begin{align}\label{e-DetReprA}
	\widehat L_{2,n}^{I,\alpha}(x) &= (\det(M))^{-1} \sum_{k=0}^n
	\det(M_k) (x-r)^\ok (x-s)^\uk\text{ or, alternatively,}\\
	\widehat L_{2,n}^{I,\alpha}(x) &=\left|
	\begin{array}{c}
	\text{(First }n\text{ rows of }M\text{)}\\
	\scriptstyle1\quad x-r \quad (x-r)(x-s) \quad (x-r)^2(x-s)\quad (x-r)^2(x-s)^2\quad \hdots \quad (x-r)^\on (x-s)^\un
	\end{array}
	\right|,
	\label{e-DetReprB}
	\end{align}
	where the $(n+1)\times (n+1)$ matrix $M$ is given by
	\begin{align*}
	&M=\\
	&\begin{pmatrix}
	\alpha & r & r(r-s)& 0& \hdots &\hdots & 0\\
	\alpha & s+\alpha(s-r) & s(s-r)& \scriptstyle s (s-r)^2 & 0& \hdots & 0\\
	{\scriptscriptstyle \frac{1}{2}\widetilde\mu_{1, 1}+\widetilde\mu_{1, 0}
		-\beta
		\widetilde\mu_{0, 0}}
	&
	{\scriptscriptstyle\frac{1}{2}\widetilde\mu_{2, 1}+\widetilde\mu_{2, 0}
		-\beta
		\widetilde\mu_{1, 0}}&
	{\scriptscriptstyle\frac{1}{2}\widetilde\mu_{2, 2}+\widetilde\mu_{2, 1}
		-\beta
		\widetilde\mu_{1, 1}}&\hdots&\hdots&\hdots&
	{\scriptscriptstyle  \frac{1}{2}\widetilde\mu_{\on+1, \un+1}+\widetilde\mu_{\on+1, \un}
		-\beta
		\widetilde\mu_{\on, \un}}\\
	\widetilde\mu_{2,1}+\widetilde\mu_{2, 0}
	&
	\widetilde\mu_{3, 1}+\widetilde\mu_{3,0}
	&
	\widetilde\mu_{3, 2}+\widetilde\mu_{3, 1}
	&\hdots&\hdots&\hdots&\widetilde\mu_{\on+2, \un+1}+\widetilde\mu_{\on+2, \un}\\
	\widetilde\mu_{2, 2}&\widetilde\mu_{3, 2}&\widetilde\mu_{3, 3}&\hdots&\hdots&\hdots&\widetilde\mu_{\on+2, \un+2}\\
	\widetilde\mu_{3, 2}&\widetilde\mu_{4, 2}&\widetilde\mu_{4, 3}&\hdots&\hdots&\hdots&\widetilde\mu_{\on+3, \un+2}\\
	\widetilde\mu_{3, 3}&\widetilde\mu_{4, 3}&\widetilde\mu_{4, 4}&\hdots&\hdots&\hdots&\widetilde\mu_{\on+3, \un+3}\\
	\vdots&\vdots&\vdots&\vdots&\vdots&\vdots&\vdots\\
	\widetilde\mu_{\on, \un}&\widetilde\mu_{\on+1, \un}&\widetilde\mu_{\on+1, \un+1}&\hdots&\hdots&\hdots&\widetilde\mu_{2\on, 2\un}\\
	\end{pmatrix}
	\end{align*}
	with $\beta = \sqrt{\alpha+1}$ and matrix $M_k$ is obtained from $M$ by replacing the $(k+1)$st column by the vector $b$.
	
	(For the precise entries in rows 3 through $n+1$, please refer to equations \eqref{e-entry3} and \eqref{e-entry4} in the proof.)
\end{theo}

\begin{rem*}
	Matrix $M$ is invertible due to the fact that the polynomials are determined precisely by the conditions that give rise to the matrix.
\end{rem*}

\begin{proof}
	First we focus on $\widehat L_{2,n}^{I,\alpha}(x) = (\det(M))^{-1} \sum_{k=0}^n
	\det(M_k) (x-r)^\ok (x-s)^\uk$. Since $Ma=b$ and by Cramer's rule we have
	\[
	a_{nk} = \frac{\det(M_k)}{\det(M)}.
	\]
	So, it suffices to show the entries of the matrix $M$.
	
	To that end we fill the first two rows of $M$ with conditions that arise from the two exceptional conditions. Namely, given the two roots $r$ and $s$ of $L_2^{\alpha-1}(-x)$, every polynomial $p$ that satisfies the eigenvalue equation for the exceptional differential expression necessarily satisfies the exceptional conditions
	\begin{align*}
	rp'(r)+\alpha p(r) &= 0,\\
	sp'(s)+\alpha p(s) &= 0
	\end{align*}
	(see Lemma \ref{lemma}).
	
	Differentiating the Ansatz in \eqref{e-Ansatz} we see
	\[
	\left(\widehat L_{2,n}^{I,\alpha}(x) \right)'
	= 
	\sum_{k=1}^n
	\ok a_{nk} (x-r)^{\ok-1} (x-s)^\uk
	+\sum_{k=2}^n \uk a_{nk} (x-r)^\ok (x-s)^{\uk-1}.
	\]
	
	And substituting this and \eqref{e-Ansatz} into the exceptional condition $rp'(r)+\alpha p(r) = 0$ yields
	\begin{align*}
	0&=
	r\left(\widehat L_{2,n}^{I,\alpha}(r) \right)'+\alpha \widehat L_{2,n}^{I,\alpha}(r)\\
	&=
	r a_{n1} + r a_{n2} (r-s)
	+
	\alpha a_{n0}\\
	&=
	\left[\alpha \qquad r \qquad r(r-s)\qquad 0\qquad \hdots \qquad 0\right]a.
	\end{align*}
	We verified the first row of the matrix $M$.
	
	For the second row, we proceed in analogy with the exceptional condition $sp'(s)+\alpha p(s) = 0$. Some more terms prevail:
	\begin{align*}
	0&=
	s\left(\widehat L_{2,n}^{I,\alpha}(s) \right)'+\alpha \widehat L_n^{I,\alpha}(s)\\
	&=
	s a_{n1} + s a_{n2} (s-r) + s a_{n3} (s-r)^2
	+
	\alpha a_{n0}
	+
	\alpha a_{n1} (s-r)\\
	&=
	\left[\alpha \qquad s+\alpha(s-r) \qquad s(s-r)\qquad s (s-r)^2 \qquad 0\qquad \hdots \qquad 0\right]a.
	\end{align*}
	We verified the second row of the matrix $M$.
	
	The remaining $n-1$ rows of $M$ (rows $3$ through $n+1$) are obtained from the orthogonality requirements
	\[
	\langle \widehat L_{2,n}^{I,\alpha}, v_l \rangle\ci{W^\alpha}
	= K_n \delta_{nl}
	\qquad
	\text{for }l=2,\hdots , n.
	\]
	Here we let $\delta_{nl}$ denote the Kronecker delta symbol. Since $v_l$ are standard only for $l\ge 4$ we begin by treating the cases $l=2$ and $l=3$ separately. To avoid confusion later on we draw attention to the fact the shift in $l$ versus which row of $M$ we are filling. That is, fixing some $l=2,\hdots , n$, we will determine the entries in the $(l+1)$st row of $M$.

	When $l=2$ we have $v_2(x)= 1/2 (x-r)(x-s)+ x-r-\sqrt{\alpha+1}$ and so with the Ansatz \eqref{e-Ansatz} we see
	\begin{align*}
	&\qquad\left\langle \widehat L_n^{I,\alpha}, v_2 \right\rangle\ci{W^\alpha}\\
	&=
	\int_0^\infty 
	\sum_{k=0}^n
	a_{nk} (x-r)^\ok (x-s)^\uk
	\left(\frac{1}{2} (x-r)(x-s)+ x-r -\sqrt{\alpha+1}\right)
	W^\alpha(x) dx\\
	&=
	\sum_{k=0}^n
	a_{nk}
	\int_0^\infty 
	\left[\frac{1}{2}(x-r)^{\ok+1} (x-s)^{\uk+1}
	+(x-r)^{\ok+1} (x-s)^{\uk}\right.
	\\
	&\qquad\qquad\qquad\quad\qquad\biggl. -\sqrt{\alpha+1}(x-r)^{\ok} (x-s)^{\uk}\biggr]
	W^\alpha(x) dx\\
	&=
	\sum_{k=0}^n
	a_{nk}
	\left[\frac{1}{2}\widetilde\mu_{\ok+1, \uk+1}+\widetilde\mu_{\ok+1, \uk}
	-\sqrt{\alpha+1}
	\widetilde\mu_{\ok, \uk}
	\right]
	\end{align*}
	by the definition of the adjusted moments in \eqref{adjmoments}. As before (in virtue of linear algebra applied the system $Ma=b$), the summands for the different values of $k$ occupy the different entries of the 3rd row of $M$:
	\begin{align}\label{e-entry3}
	M_{3,k+1} = \frac{1}{2}\widetilde\mu_{\ok+1, \uk+1}+\widetilde\mu_{\ok+1, \uk}
	-\sqrt{\alpha+1}
	\widetilde\mu_{\ok, \uk}
	\qquad\text{for }0\le k\le n.
	\end{align}
	Again, there is a shift between the value of $k$ and the column of $M$.
	For example, for $k=0$ we obtain the $(3,1)$ entry of $M$, $M_{3,1}$, to equal
	\[
	M_{3,1} 
	= \frac{1}{2}\widetilde\mu_{1, 1}+\widetilde\mu_{1, 0}
	-\sqrt{\alpha+1}
	\widetilde\mu_{0, 0}.
	\]
	For $k=1$ and $k=2$ we see
	\begin{align*}
	k=1: \qquad M_{3,2} &=
	\frac{1}{2}\widetilde\mu_{2, 1}+\widetilde\mu_{2, 0}
	-\sqrt{\alpha+1}
	\widetilde\mu_{1, 0},\\
	k=2: \qquad M_{3,3} &=
	\frac{1}{2}\widetilde\mu_{2, 2}+\widetilde\mu_{2, 1}
	-\sqrt{\alpha+1}
	\widetilde\mu_{1, 1},
	\end{align*}
	and so forth.
	We verified the third row of $M$.
	
	Let us focus on the fourth row of $M$. To do so we take $l=3$ and with our choice $v_3(x) = (x-r)^2(x-s)+(x-r)^2$ for the degree three flag element we see
	\begin{align*}
	&\qquad\left\langle \widehat L_{2,n}^{I,\alpha}, v_3 \right\rangle\ci{W^\alpha}\\
	&=
	\int_0^\infty 
	\sum_{k=0}^n
	a_{nk} (x-r)^\ok (x-s)^\uk
	\left[(x-r)^2(x-s)+(x-r)^2\right]
	W^\alpha(x) dx\\
	&=
	\sum_{k=0}^n
	a_{nk}
	\int_0^\infty 
	\left[(x-r)^{\ok+2} (x-s)^{\uk+1}
	+(x-r)^{\ok+2} (x-s)^{\uk}\right]
	W^\alpha(x) dx\\
	&=
	\sum_{k=0}^n
	a_{nk}
	\left[\widetilde\mu_{\ok+2, \uk+1}+\widetilde\mu_{\ok+2, \uk}
	\right].
	\end{align*}
	Again for $k=0, 1, 2$ we obtain the matrix entries
	\begin{align*}
	k=0: \qquad M_{4,1} &=\widetilde\mu_{\ok+2, \uk+1}+\widetilde\mu_{\ok+2, \uk}
	\,\,\stackrel{k=0}{=}\,\,
	\widetilde\mu_{2,1}+\widetilde\mu_{2, 0}
	,\\
	k=1: \qquad M_{4,2} &=\widetilde\mu_{\ok+2, \uk+1}+\widetilde\mu_{\ok+2, \uk}
	\,\,\stackrel{k=1}{=}\,\,
	\widetilde\mu_{3, 1}+\widetilde\mu_{3,0}
	,\text{ and}\\
	k=2: \qquad M_{4,3} &
	=\widetilde\mu_{\ok+2, \uk+1}+\widetilde\mu_{\ok+2, \uk}
	\,\,\stackrel{k=2}{=}\,\,
	\widetilde\mu_{3, 2}+\widetilde\mu_{3, 1}.
	\end{align*}
	This shows the entries in the fourth row of $M$.
	
	When the degree of the exceptional polynomial $n\ge 4$, the rows five through $n+1$ we are obtained from the standard flag elements $$v_l = (x-r)^\ol(x-s)^\ul\, ,$$
	for $l=4, \hdots, n$. And so for those values of $l$ we have
	\begin{align*}
	\qquad\left\langle \widehat L_{2,n}^{I,\alpha}, v_l \right\rangle\ci{W^\alpha}
	&=
	\int_0^\infty 
	\sum_{k=0}^n
	a_{nk} (x-r)^\ok (x-s)^\uk
	\left[(x-r)^\ol(x-s)^\ul\right]
	W^\alpha(x) dx\\
	&=
	\sum_{k=0}^n
	a_{nk}
	\int_0^\infty 
	(x-r)^{\ok+\ol} (x-s)^{\uk+\ul}
	W^\alpha(x) dx
	=
	\sum_{k=0}^n
	a_{nk}
	\widetilde\mu_{\ok+\ol, \uk+\ul}.
	\end{align*}
	This results in the $(l+1,k+1)$ matrix entry
	\begin{align}\label{e-entry4}
	M_{l+1,k+1} =\widetilde \mu_{\ok+\ol, \uk+\ul}
	\qquad\text{for }3\le l \le n, 0\le k\le n.
	\end{align}
	This concludes the proof for the entries of the matrix $M$, and thereby the proof of equation \eqref{e-DetReprA}.
	
	%
	
	To obtain the second representation, \eqref{e-DetReprB}, we notice that by the definition of the vector $b$ in \eqref{d-b} we have
	\[
	\left|
	\begin{array}{c}
	\text{(First }n\text{ rows of }M\text{)}\\
	0 \quad \hdots \quad 0 \quad (x-r)^\ok(x-s)^\uk\quad 0 \quad \hdots \quad 0
	\end{array}
	\right|
	=
	\frac{\det(M_k)}{K_n} \,(x-r)^\ok(x-s)^\uk
	\]
	for $k=0,1,\hdots, n$. In the last row of the matrix, the entry $(x-r)^\ok(x-s)^\uk$ is in the $(n+1, k)$ position. Formula \eqref{e-DetReprB} now follows simply by co-factor expansion of the determinant in \eqref{e-DetReprB} along the last row.
\end{proof}

\section{Moment Recursion Formulas}\label{s-moments}

As the matrix $M$ used in the determinantal representation of the polynomials 
$\widehat L_{2,n}^{I,\alpha}(x)$ is largely populated by various adjusted moments, we need to develop a suite of recursion-like formulas which will allow us to compute all entries in the two-dimensional array of these moments.

\begin{theo}
\label{t-Recursion}
The adjusted moments $\wt\mu_{i,j}= \int_0^\infty (x-r)^i (x-s)^j \, \widehat W^{I,\alpha} (x) \, dx$
\begin{itemize}

\item[(a)] satisfy the three-term recursion-like formula
\begin{align} \label{3-term}
  \widetilde{\mu}_{i+1,j} = \; \widetilde{\mu}_{i,j+1} + 2\sqrt{\alpha+1} \, 
  \widetilde{\mu}_{i,j} \hspace{5.6 cm}(i,j\in \N_0),
\end{align}

\hspace{-0.57 cm} as well as

\item[(b)] the four-term recursion-like formula
\begin{align} \label{4-term}
  \widetilde{\mu}_{i+1,j+1} = \; &[i+j-1+2\alpha+\sqrt{\alpha+1}] \, 
  \widetilde{\mu}_{i+1,j}
  \\ \notag + \; &[(1-i-j)(\alpha+1)+(3-3i-j-4\alpha)\sqrt{\alpha+1}] \, \widetilde{\mu}_{i,j} 
  \\ \notag + \; &[(2i-4)(\alpha+1)(\sqrt{\alpha+1}+1)] \, \widetilde{\mu}_{i-1,j}
  \hspace{2.7 cm}(i \in \N \mbox{ and } j\in \N_0).
\end{align}

\end{itemize}
\end{theo}

Throughout the observations and proof below, we simplify notation by writing $W$ for $\widehat W_2^{I,\alpha}$.

The following general observations will be key to our proof of part (b).  To obtain this result, we make use of several facts.  First, for functions $f,g$ which are smooth on $[0,\infty)$ the moment functionals satisfy:
\[ \left\langle W', f \right\rangle = - \left\langle W, f' \right\rangle \quad
\text{and}\quad \left\langle g W, f \right\rangle = \left\langle W, fg \right\rangle, \]
where $\langle\fdot, \fdot\rangle$ denotes the inner product with respect to Lebesgue measure on $[0,\infty)$.

Next, for a linear operator of the form $$\ell[y] = a_2y'' + a_1y' + a_0y,$$ the related symmetry equation is given by $$a_2y' + (a'_2 - a_1)y = 0,$$ and it is solved by the weight function (with respect to which the eigenpolynomials are orthogonal) \cite{Littlejohn1984}.
That is,
\[ a_2W' + (a_2' - a_1)W=0. \]

And together with
\begin{align*}
\left\langle a_2W', \; (x-r)^i (x-s)^j \right\rangle
&= -\left\langle W, \; a_2'(x-r)^i (x-s)^j \right\rangle 
  - i\left\langle W, \; a_2(x-r)^{i-1} (x-s)^j \right\rangle \\
& \hspace{0.46 cm} - j\left\langle W, \; a_2(x-r)^i (x-s)^{j-1} \right\rangle,  
\end{align*}
where the goal is to convert $W'$ into $W$, we find for $k\in \N$:

\begin{align}
  \notag 0 &= \left\langle a_2W' + (a_2' - a_1)W, \; (x-r)^i (x-s)^j
  \right\rangle\\ \label{eq:innerprod}
  &= i\left\langle W, \; a_2(x-r)^{i-1} (x-s)^j \right\rangle + j\left\langle W, \; 
  a_2(x-r)^i (x-s)^{j-1} \right\rangle \\ \notag
  & \hspace{0.41 cm} + \left\langle W, \; a_1(x-r)^i (x-s)^j
  \right\rangle.
\end{align}

\begin{proof}
  The proof of part (a) will simply rely on the relation $r = s - 2\sqrt{\alpha+1}$, where $r$ and $s$ are the exceptional roots first mentioned in Section 3.  We obtain
\begin{align*}
  \widetilde{\mu}_{i+1,j} 
  &= \int_0^\infty (x-r)^{i+1} (x-s)^j \, W(x) \, dx \\
  &= \int_0^\infty (x-r)^{i} (x-s+2\sqrt{\alpha+1}) (x-s)^j \, W(x) \, dx \\
  &= \int_0^\infty (x-r)^i (x-s)^{j+1} \, W(x) \, dx + 2\sqrt{\alpha+1} \int_0^\infty (x-r)^i (x-s)^j \, W(x) \, dx \\
  &= \widetilde{\mu}_{i,j+1} + 2\sqrt{\alpha+1} \widetilde{\mu}_{i,j}.
\end{align*}

  To start the proof of part (b), we note that the differential expression \eqref{eq: xop-de} for the $X_2^I$-Laguerre XOP system is
\begin{align*} \widehat
\ell_2^{\text{I,}\alpha}[y](x)= -xy^{\prime\prime} + \left(x - \alpha -1 + 
2x\frac{(L_2^{\alpha-1}(-x))^{\prime}}{L_2^{\alpha-1}(-x)}\right) y^{\prime} +
\left(\frac{2\alpha(L_2^{\alpha-1}(-x))^{\prime}} {L_2^{\alpha-1}(-x)} - 2\right) y,
\end{align*}
where the classical Laguerre polynomial $L_2^{\alpha-1}(-x) = \frac{1}{2}x^2 + (\alpha+1)x + \frac{\alpha(\alpha+1)}{2}$, and $(L_2^{\alpha-1}(-x))^{\prime} = x + \alpha + 1$.

For this differential expression, the coefficient $a_2(x) = -x$ and, after a slight simplification,
\begin{align*}
  a_1(x) = x - \alpha - 1 + \frac{4x(x+\alpha+1)}{x^2 + 2(\alpha+1)x + \alpha(\alpha+1)}.
\end{align*}

In order to have adjusted moments appear from these calculations, we re-write both $a_1$ and $a_2$ so that $x$ only appears in powers of the linear factors $(x-r)$ and $(x-s)$.

Accordingly, $a_2(x)$ can be represented as $-(x-r)-r$, and we find that
\begin{align*}
  a_1(x) = (x-r) + A - \frac{B}{(x-r)} + \frac{C}{(x-r)(x-s)},
\end{align*}
where $A = -2\alpha + 2 - \sqrt{\alpha+1}, B = 4(\alpha+1),$ and $C = (\alpha+1)(1-\sqrt{\alpha+1}),$ and where we have arbitrarily given priority to the factor $(x-r)$ in these representations.

Substituting $a_1$ and $a_2$ into \eqref{eq:innerprod}, we get

\begin{align*}
  0 &= i\left\langle W, \; [-(x-r)-r](x-r)^{i-1} (x-s)^j \right\rangle + j\left\langle W, \; [-(x-r)-r](x-r)^i (x-s)^{j-1} \right\rangle \\ & \hspace{0.46 cm} + \left\langle W, \; \left[(x-r) + A - \frac{B}{(x-r)} + \frac{C}{(x-r)(x-s)}\right](x-r)^i (x-s)^j \right\rangle \\
  &= -i\widetilde{\mu}_{i,j} - ir\widetilde{\mu}_{i-1,j} 
     - j\widetilde{\mu}_{i+1,j-1} - jr\widetilde{\mu}_{i,j-1}
     + \widetilde{\mu}_{i+1,j} + A\widetilde{\mu}_{i,j}
     - B\widetilde{\mu}_{i-1,j} + C\widetilde{\mu}_{i-1,j-1}\,,
\end{align*}  where the coefficients $i$ and $j$ arise from the subscript of the moment $\mu_{i,j}$.  ($i$ does not represent the imaginary number $\sqrt{-1}$.)
Thus,
\begin{align*}
  \widetilde{\mu}_{i+1,j} = j\widetilde{\mu}_{i+1,j-1} 
   + (i-A)\widetilde{\mu}_{i,j} + (ir+B)\widetilde{\mu}_{i-1,j}
   + jr\widetilde{\mu}_{i,j-1} - C\widetilde{\mu}_{i-1,j-1}.
\end{align*}

After shifting the index $j \mapsto j+1$, we have
\begin{align}
  \label{eq:recur1}
  \widetilde{\mu}_{i+1,j+1} &= (j+1)\widetilde{\mu}_{i+1,j} 
   + (i-A)\widetilde{\mu}_{i,j+1} + (ir+B)\widetilde{\mu}_{i-1,j+1} \\
   \notag & \hspace{0.42 cm} + (j+1)r\widetilde{\mu}_{i,j} - C\widetilde{\mu}_{i-1,j}.
\end{align}

Using the identity \eqref{3-term} of part (a) for $\widetilde{\mu}_{i,j+1}$ and $\widetilde{\mu}_{i-1,j+1}$, \eqref{eq:recur1} may be reduced from five to three terms on the right-hand side.  After some simplifying, we have
\begin{align*}
  \widetilde{\mu}_{i+1,j+1} &= (i+j+1-A)\, \widetilde{\mu}_{i+1,j}
  + [(i-A)(-2\sqrt{\alpha+1}) + (ir+B) + (j+1)r] \,  \widetilde{\mu}_{i,j} \\ & \hspace{0.46 cm} + [(ir+B)(-2\sqrt{\alpha+1}) - C] \, \widetilde{\mu}_{i-1,j}.
\end{align*}

Finally, after substituting in the values for $A, B, C,$ and $r$, we arrive at the result of part (b).

\end{proof}

\begin{rem*}

By switching the priority from the linear factor $(x-r)$ to $(x-s)$, but otherwise exactly following the steps of the proof of part (b) of Theorem \ref{t-Recursion}, we obtain a second four-term recursion-like formula
\begin{align} \label{4-term_b}
  \widetilde{\mu}_{i+1,j+1} = \; &[i+j-1+2\alpha-\sqrt{\alpha+1}] \, 
  \widetilde{\mu}_{i,j+1}
  \\ \notag & \hspace{-0.07 cm} + \; [(1-i-j)(\alpha+1)+(-3+i+3j+4\alpha)\sqrt{\alpha+1}] \, \widetilde{\mu}_{i,j} 
  \\ \notag & \hspace{-0.07 cm} + \; [(-2j+4)(\alpha+1)(\sqrt{\alpha+1}-1)] \, \widetilde{\mu}_{i,j-1}
  \hspace{2.7 cm}(i \in N \mbox{ and } j\in \N_0).
\end{align}

\end{rem*}

We have little doubt that numerous recursion-like relationships exist for the adjusted moments.  However, we will see in the next section that the three formulas given above in \eqref{3-term}, \eqref{4-term}, and \eqref{4-term_b}, in conjunction with three initial values, $\widetilde{\mu}_{2,2}$, $\widetilde{\mu}_{1,2}$, and $\widetilde{\mu}_{2,1}$, will suffice to compute all of the adjusted moments.

\section{Initial moments} \label{App:AppendixA}

\begin{theo}
\label{t-Init_Mom}

The adjusted moments for the $X_2^I$-Laguerre XOP are completely determined by the recursion-like formulas of Theorem \ref{t-Recursion}, and the three initial values
\begin{align*}
  \widetilde{\mu}_{2,2} &= 4\Gamma(1+\alpha), \\
  \widetilde{\mu}_{1,2} &= 4e^{-r} (-r)^\alpha \Gamma(1+\alpha) \Gamma(-\alpha,-r)
  ,\text{ and}\\
  \widetilde{\mu}_{2,1} &= 4e^{-s} (-s)^\alpha \Gamma(1+\alpha) \Gamma(-\alpha,-s),
\end{align*}
where the Gamma function is given by $\Gamma(x):=\int_0^\infty t^{x-1} e^{-t} dt$, and the incomplete Gamma function by $\Gamma(x,a) := \int_a^\infty t^{x-1} e^{-t} dt$ for $a>0$.

\end{theo}

\begin{proof}

Recall that $L_2^{\alpha-1}(-x) = \frac{1}{2}(x-r)(x-s)$.  The first adjusted moment given above, $\widetilde{\mu}_{2,2}$, is the simplest to compute because of the complete cancellation of the linear factors $(x-r)$ and $(x-s)$ that occurs.  We have
\begin{align*}
  \widetilde{\mu}_{2,2} &= \int_0^\infty (x-r)^2 (x-s)^2 \, W^\alpha (x) \, dx = \int_0^\infty (x-r)^2 (x-s)^2 \, \frac{x^{\alpha}e^{-x}}{(L_2^{\alpha-1}(-x))^2} \, dx \\
  &= 4\int_0^\infty x^{\alpha}e^{-x} \, dx = 4\Gamma(1+\alpha).
\end{align*}

For $\widetilde{\mu}_{1,2}$ and $\widetilde{\mu}_{2,1}$, we rely on the result of Theorem 4.1 in \cite{Liaw-Osborn}, where it is established that
\begin{align*}
  \int_0^\infty \frac{e^{-x} x^\beta}{(x+\alpha)} dx 
  &= e^\alpha\, E_{1+\beta}(\alpha)\, \Gamma(1+\beta),\quad
  \alpha>0,\beta >-1, \quad \text{where}
\end{align*}

\[ E_a(x) = \int_1^\infty e^{-x t} t^{-a} dt = x^{a-1} \Gamma(1-a,x), \quad x>0. \]

This result allows us to compute both
\begin{align*}
  \widetilde{\mu}_{1,2} &= \int_0^\infty (x-r) (x-s)^2 \, W^\alpha (x) \, dx = 4\int_0^\infty \frac{x^{\alpha}e^{-x}}{(x-r)} \, dx \\
  &= 4e^{-r}E_{1+\alpha}(-r)\, \Gamma(1+\alpha) = 4e^{-r} (-r)^\alpha \Gamma(1+\alpha) \Gamma(-\alpha,-r), \\
  &\hspace{-2.74 cm} \text{and similarly, by symmetry,} \\
  \widetilde{\mu}_{2,1} &= 4e^{-s} (-s)^\alpha \Gamma(1+\alpha) \Gamma(-\alpha,-s).
\end{align*}

To see that these three adjusted moments are sufficient to start a recursion-like process that can be used to calculate all other moments, consider the four diagrams in Figures \ref{fig:fig1} and \ref{fig:fig2}.

Figure \ref{fig:fig1}(a) represents the behavior of the three-term recursion-like formula given in \eqref{3-term}.  Namely, if we know the values of any two of the three adjusted moments marked by the symbol *, the formula allows us to compute the third moment.  (In this and the subsequent figures, the first subscript of the adjusted moments is indicated going down the left side of the figure, while the second subscript is given along the top.)

Figure \ref{fig:fig1}(b) represents the behavior of the four-term recursion-like formula given in \eqref{4-term_b}.  That is, if we know the values of any three of the four adjusted moments marked by the symbol *, the formula allows us to compute the fourth moment.

And Figure \ref{fig:fig1}(c) represents the behavior of the four-term recursion-like formula given in \ref{4-term}.  As before, if we know the values of any three of the four adjusted moments marked by the symbol *, the formula allows us to compute the fourth moment.

Finally, Figure \ref{fig:fig2} demonstrates the first few steps in starting to compute all entries of the two dimensional array of adjusted moments.  The cells marked ``A" represent the 3 adjusted moments $\widetilde{\mu}_{2,2}, \, \widetilde{\mu}_{1,2},$ and $\widetilde{\mu}_{2,1}$ calculated at the start of this theorem.  Next, $\widetilde{\mu}_{1,1}$, marked with a ``B", can be computed using the three-term formula represented in Figure \ref{fig:fig1}(a), along with the known values of $\widetilde{\mu}_{1,2}$ and $\widetilde{\mu}_{2,1}$.

Then we find one of the two adjusted moments marked with a ``C" by using the four-term formula displayed in Figure \ref{fig:fig1}(b), and we get the other moment marked ``C" from the other four-term formula, as demonstrated in Figure \ref{fig:fig1}(c).  To compute the adjusted moment marked ``D", we once again rely on the three-term formula, using the now-known values of $\widetilde{\mu}_{0,1},$ and $\widetilde{\mu}_{1,0}$.

Finally, the values of $\widetilde{\mu}_{0,2},$ and $\widetilde{\mu}_{2,0}$ are computed making use of the two four-term formulas---one for each.  Once the 3-by-3 sub-array at the upper left corner of the infinite array of adjusted moments has been filled out, it is clear that the remaining entries can all be computed using the trio of recursion-like formulas---\eqref{3-term}, \eqref{4-term}, and \eqref{4-term_b}.


\begin{figure}
	\footnotesize
	\begin{minipage} {1.00\textwidth}
		\centering
	\begin{tabular}{c|C{.3 in}|C{.3 in}|C{.4 in}c|C{.3 in}|C{.3 in}|C{.3 in}|C{.4 in}c|C{.3 in}|C{.3 in}|}
		  \multicolumn{1}{r}{}
		& \multicolumn{1}{c}{$j$}
		& \multicolumn{1}{c}{$j+1$}
		& \multicolumn{1}{r}{}
		& \multicolumn{1}{r}{}
		& \multicolumn{1}{c}{$j-1$}
		& \multicolumn{1}{c}{$j$} 
		& \multicolumn{1}{c}{$j+1$}
		& \multicolumn{1}{r}{}
		& \multicolumn{1}{r}{}
		& \multicolumn{1}{c}{$j$}
		& \multicolumn{1}{c}{$j+1$} \\
   		  \cline{2-3} \cline{6-8} \cline{11-12}
		$i-1$ & & & & $i-1$ & & & & & $i-1$ & * & \\
		\cline{2-3} \cline{6-8} \cline{11-12}
		$i$ & * & * & & $i$ & * & * & * & & $i$ & * & \\
		\cline{2-3} \cline{6-8} \cline{11-12}
        $i+1$ & * & & & $i+1$ & & & * & & $i+1$ & * & * \\
	    \cline{2-3} \cline{6-8} \cline{11-12}
	\end{tabular}
	\caption{}\label{fig:fig1}
	\hspace{1.0 cm} (a) \hspace{4.55 cm} (b) \hspace{4.65 cm} (c)
	\end{minipage}
\end{figure}

\begin{figure}
\caption{Filling out array of adjusted moments}
\label{fig:fig2}
	\begin{tabular}{ C{.3 in}|C{.3 in}|C{.3 in}|C{.3 in}|C{.3 in}| }
		\multicolumn{1}{r}{$i \bigg \backslash j$}
		&  \multicolumn{1}{c}{$0$}
		& \multicolumn{1}{c}{$1$} 
		& \multicolumn{1}{c}{$2$} 
		& \multicolumn{1}{c}{$\cdots$}\\
		\cline{2-5}
		$0$& D & C& $\cdots$ & $\cdots$ \\
		\cline{2-5}
		$1$& C & B & A & $\cdots$ \\
		\cline{2-5}
		$2$ & $\vdots$ & A & A &$\cdots$ \\
		\cline{2-5}
		$\vdots$ & $\vdots$  & $\vdots$  &$\vdots$  & $\ddots$\\
		\cline{2-5}
	\end{tabular}
\end{figure}


\end{proof}


\begin{thebibliography}{99}                                                                                               %
	
	%
	\bibitem {Atia-Littlejohn-Stewart}M.~Atia, L.L.~Littlejohn, and J.~Stewart, \emph{\ The Spectral Theory of the $X_{1}$-Laguerre Polynomials}, Adv.~Dyn.~Syst.~Appl., 8(2) (2013) 81--92.
	
	%
	
	\bibitem{Chihara} T.~Chihara, \emph{An Introduction to Orthogonal Polynomials,} Gordon and Breach, Science Publisher Inc., New York, 1978. Unabridged republication by Dover Publications, 2011.
	
	%
	%
	%
	
	\bibitem {Dutta-Roy}D.~Dutta and P.~Roy, \emph{Conditionally exactly solvable potentials and exceptional orthogonal polynomials, }J.~Math.~Phys.~51 (2010) 042101.
	
	%
	
	\bibitem{Classification} M.~\'{A}ngeles Garc\'{i}a--Ferrero, D.~G\'{o}mez--Ullate, and R.~Milson. \emph{A Bochner type classification theorem for exceptional orthogonal polynomials.} ArXiv e-prints, March 2016.
	
	%
	%
	\bibitem {KMUG}D.~G\'{o}mez--Ullate, N.~Kamran, and R.~Milson, \emph{An extended class of orthogonal polynomials defined by a Sturm-Liouville problem}, J.~Math.~Anal.~Appl.~359 (2009) 352--367.
	
	\bibitem {KMUG1}D.~G\'{o}mez--Ullate, N.~Kamran, and R.~Milson, \emph{An extension of Bochner's problem: exceptional invariant subspaces}, J.~Approx.~Theory 162 (2010) 987--1006.
	
	\bibitem {KMUG9}D.~G{\'{o}}mez--Ullate, N.~Kamran, and R.~Milson, \emph{Exceptional orthogonal polynomials and the Darboux transformation}, J.~Phys.~A: Math.~Theor.~43 (2010) 434016.
	
	\bibitem {KMUG3}D.~G\'{o}mez--Ullate, N.~Kamran, and R.~Milson, \emph{Two-step Darboux transformations and exceptional Laguerre polynomials}, J.~Math.~Anal.~ Appl., 387 (2012) 410--418.
	
	\bibitem {KMUG4}D.~G\'{o}mez--Ullate, N.~Kamran, and R.~Milson, \emph{Structure theorems for linear and non-linear differential operators admitting invariant polynomial subspaces}, Discrete Contin.~Dyn.~Syst.~18 (2007) 85--106.
	
	\bibitem {KMUG5}D.~G\'{o}mez--Ullate, N.~Kamran, and R.~Milson, \emph{A conjecture on exceptional orthogonal polynomials},	Found.~Comput.~Math.~13 (2013) 615--666.
	
	\bibitem {KMUG6}D.~G\'{o}mez--Ullate, N.~Kamran, and R.~Milson, \emph{On orthogonal polynomials spanning a non-standard flag,} Contemp.~Math., 563 (2012) 51--71.
	
	\bibitem{Recurrence} D. G\'{o}mez-Ullate, A. Kasman, A.B.J. Kuijlaars, R. Milson, \emph{Recurrence relations for exceptional Hermite polynomials}, J. Approx. Theory, 204 (2016) 1--16.
	
	\bibitem {Zeros}D.~G\'{o}mez--Ullate, F.~Marcell\'{a}n, and R.~ Milson, \emph{Asymptotic and interlacing properties of zeros of exceptional Jacobi and Laguerre polynomials}, J.~Math.~Anal.~Appl.~399 (2013) 480--495.
	
	\bibitem {grandati11}Y.~Grandati, \emph{Solvable rational extensions of the isotonic oscillator}, Ann.~Phys.~326, (2011) 2074--2090.
	
	\bibitem{grandati12} Y.~Grandati, \emph{Rational extensions of solvable potentials and exceptional orthogonal polynomials}, J.~Physics, 343, (2012) 012041.
	
	%
	%
	
	\bibitem {HoOdakeSasaki}C.L.~Ho, S.~Odake, and R.~Sasaki, \emph{Properties of the exceptional }$(X_{\ell})$ \emph{Laguerre and Jacobi polynomials, }SIGMA, vol.~7, article 107, (2011) 24 pages.
	
	\bibitem {HoSasakiZeros2012}C.L.~Ho and R.~Sasaki, \emph{Zeros of the Exceptional Laguerre and Jacobi Polynomials}, ISRN Mathematical Physics, Article ID 920475 (2012) 27 pages.
	
	\bibitem{Kelly-Liaw-Osborn} J.S.~Kelly, C.~Liaw and J.~Osborn \emph{Moment representations of exceptional $X_1$ orthogonal polynomials,} in print by Mathematische Nachrichten (early view available online January 4th, 2017).
	
	\bibitem {LLMS} C.~Liaw, L.~Littlejohn, R.~Milson, and J.~Stewart, \emph{The spectral analysis of three families of exceptional Laguerre polynomials,} J.~Approx.~Theory 202 (2016) 5--41.
	
	%
	%
	
	\bibitem{Liaw-Osborn} C.~Liaw and J.~Osborn \emph{Moment representations of the exceptional $X_1$-Laguerre orthogonal polynomials,} accepted for publication by J.~Math.~Anal.~Appl.
	
	\bibitem {Littlejohn1984} L.L.~Littlejohn, \emph{On the classification of Differential Equations Having Orthogonal Polynomials Solutions,} Annali di Matematica pura ed applicata (IV), vol.~CXXXVIII (1984), 35-53.
	
	
	\bibitem {Midya-Roy}B.~Midya and B.~Roy, \emph{Exceptional orthogonal polynomials and exactly solvable potentials in position dependent mass Schr\"{o}dinger Hamiltonians, }Phys.~Lett.~A 373 (45) (2009) 4117--4122.
	
	
	\bibitem{Odake-Sasaki1} S.~Odake and R.~Sasaki, \emph{Another set of infinitely many exceptional $(X_{\ell})$ Laguerre polynomials}, Physics Letters B 684 (2009) 414--417.
	
	\bibitem {Quesne}C.~Quesne, \emph{Exceptional orthogonal polynomials, exactly solvable potentials and supersymmetry, }J.~Phys.~A: Math.~Theor.~41 (2008) 392001.
	
	\bibitem {Quesne2}C.~Quesne, \emph{Solvable rational potentials and exceptional orthogonal polynomials in supersymmetric quantum mechanics, }SIGMA, vol.~5, article 084, (2009) 24 pages.
	
\end{thebibliography}
\end{document}